\documentclass{article}
\usepackage{maa-monthly}

 \newcommand{\N}{\mathbb{N}}

 \newcommand{\beq}{\begin{eqnarray*}}
 \newcommand{\eeq}{\end{eqnarray*}}

\theoremstyle{theorem}
\newtheorem{thm}{Theorem}

\newtheorem{lem}[thm]{Lemma}
\newtheorem{prop}[thm]{Proposition}

\theoremstyle{definition}
\newtheorem{defn}[thm]{Definition}
\numberwithin{equation}{section}

\begin{document}
\title{A Novel Proof for Kimberling's Conjecture on Doubly Fractal Sequences}
\markright{A Novel Proof for Kimberling's Conjecture}
\author{Matin Amini and Majid Jahangiri}

\maketitle

\begin{abstract}
A sequence is a fractal sequence if it contains itself as a proper subsequence. (The self-containment property resembles that of visual fractals) A doubly fractal sequence of integers is defined by operations called upper trimming and lower trimming. C. Kimberling proved that signature sequences are doubly fractal and conjectured the converse. This article gives a procedure for constructing doubly fractal sequences and proves Kimberling's conjecture.
\end{abstract}


\section{Introduction}
Geometric fractals are characterized by visual self-similarity. That is, zooming on a visual fractal shows that the original structure is repeated inside itself infinitely many times. Clark Kimberling noticed that a sequence of integers can also contain infinitely many copies of itelf in two different ways \cite{4}, and he studied properties of such sequences \cite{1,2,3,5,6}. Briefly, operations called upper trimming and lower trimming, applied to suitable sequences $S$, leave behind the sequence $S$ itself. An integer sequence of this sort is called a doubly fractal sequence. Signature sequences as defined in Section \ref{Def}, are known to be doubly fractal. In this article we introduce a procedure for constructing doubly fractal sequences, and then we apply the procedure to prove Kimberling's conjecture that the signature sequences are the {\it only} doubly fractal sequences.\\
The structure of this article is as follows: definitions and required preliminaries are given in Section \ref{Def}. The constructive procedure is given in Section \ref{The Construction}, and the relationship between the doubly fractal sequences and the signature sequences is discussed in Section \ref{Rel}.

\section{Definitions and Preliminaries}\label{Def}
Suppose $ S=(s_1, s_2, s_3, \ldots)$ is a finite or infinite sequence of numbers in the set $\N$ of positive integers. The $k$th term of a sequence $S$ is denoted by $S(k)$ or $s_k$.
\begin{defn} Suppose every positive integer occurs in $S$. The upper trimmed (sub)sequence of $S$, denoted by $\wedge_S$, is the sequence that remains after every first occurrence of every positive integer is removed from $S$.
\end{defn}
\begin{defn}
Suppose every positive integer occurs in $S$. The lower trimmed (sub)sequence of $S$, denoted by $\vee_S$, is the sequence that remains after $1$ is subtracted from every term of $S$ and then all $0$s are removed.
\end{defn}
\begin{defn}\label{defn3}
A sequence $S$ is a doubly fractal sequence if $s_1=1$ and $\wedge_S=\vee_S=S$.
\end{defn}
\begin{defn}
A sequence $(1,2,\ldots,n,m,\ldots)$ in $\N$ where $m \neq n$ is an {\it initial segment of type 1}. A sequence $(1,1,\ldots,1,m,\ldots)$ in $\N$, where $m>1$, is an {\it initial segment of type 2}. 
\end{defn}
Now suppose every positive integer occurs at least once in $S$. The index of the $k$th occurrence of $n$ is denoted by $I_k(n)$.
\begin{lem}\label{4}
Suppose $S$ is a doubly fractal sequence. Then $S$ has an initial segment of type 1 or 2, and, {\it a fortiori}, that segment is one of these two forms:
$(1,2,3,\ldots n,1)$ or $(\underset{n}{\underbrace{1,1,\ldots,1}},2)$.
\end{lem}
\begin{proof} First $s_1=1$ by Definition \ref{defn3}. We continue with two cases.
Case 1: $s_2 \neq 1$. If $s_2>2$, then $\vee_S(1)=s_2-1>1 \neq 1$, so that necessarily $s_2=2$. Let $n$ be the greatest integer $k$ for which $s_k=k$. If $s_{n+1} \geq n+1$, then $\wedge_S(n) \geq n+1 \neq n=s_n$, contrary to $\wedge_S=S$. Therefore, $s_{n+1}=1$ as desired.\\
Case 2: $s_2=1$. Suppose $s_k=1$ for $k=1,2,\ldots,n$ and $s_{n+1} \neq n$. If $s_{n+1}>2$, then $\vee_S(1)=s_{n+1}-1>1=s_1$, contrary to $\vee_S=S$. Therefore $s_{n+1}=2$. 
\end{proof}
\noindent The sequence below is an example of a doubly fractal sequence which the second term is not 2.
\beq
S=(1,1,1,2,1,2,1,2,3,1,2,3,1,2,\ldots).
\eeq

\begin{defn}
Let $\theta$ be a real number and let $M_\theta=\{i+j\theta \hbox{ } : \hbox{ } i, j \in \N \}$, a multiset depending on $\theta$. We arrange the numbers in $M_{\theta}$ in nondecreasing order(with duplicates if and only if $\theta$ is rational) so that the ordered set can be represented as
\beq
M_{\theta}=\left(s_h+a_h\theta\right)_{h=1}^\infty.
\eeq
The sequence $S_\theta=(s_1,s_2,s_3,\ldots)$ is the {\it signature (sequence) of $\theta$}. For example,
\begin{align*}
S_{\sqrt13}=(1,2,3,4,1,5,2,6,3,7,4,8,1,5,9,2,6,10,3,7,11,4,8\ldots),
\end{align*}
and the signature of $1/7$ is
\begin{align*}
S_{1/7}=(1,1,1,1,1,1,1,2,1,2,1,2,1,2,1,2,1,2,1,2,1,3,2,1\ldots).
\end{align*}
\end{defn}
The next proposition is proved in \cite{1}.
\begin{prop}\label{dsprop}
The signature of a positive number $\theta$ is a doubly fractal sequence.
\end{prop}

\begin{thm}\label{6}
Distinct numbers have distinct signatures.
\end{thm}
\begin{proof}
Prove this by contradiction. Without loose of generality, assume that $\alpha<\beta$. Denote the decimal representation of these numbers as $\alpha=a.a_1a_2a_3\ldots a_m\ldots$ and $\beta=b.b_1b_2b_3\ldots b_m\ldots$, in which $a_k=b_k$, for all $k\leq m$ and $a_{m+1}\neq b_{m+1}$. Now for every $e,f,g,h \in \N$ we have $e+f\alpha<g+h\alpha$ if and only if $e+f\beta<g+h\beta$. Equivalently $g-e+(h-f)\alpha>0$ if and only if $g-e+(h-f)\beta>0$. So we can write $(-bb_1b_2b_3\ldots b_mb_{m+1})+\beta10^{m+1} \geq 0$ but we have $(-bb_1b_2b_3\ldots b_mb_{m+1})+\alpha10^{m+1}<0$, a contradiction.
\end{proof}
\section{The Construction Procedure}\label{The Construction}
In this section we explain a procedure which constructs doubly fractal sequences. The initial segment of a doubly fractal sequence $S$ is either $(1,2,3,\ldots,n,1)$ or $(\underset{n}{\underbrace{1,1,\ldots,1}},2)$,  by Lemma \ref{4}. We assume the former case, and the latter is translated to the former. Start with $S=(1,2,3,\ldots,n)$, we extend this sequence in a way that in each step the result is always a doubly fractal (finite) sequence. The terms of this initial segment is referred by {\it the main terms}. A subsequence starts from a $1$ term until the next $n$ term is called a {\it block}. In each step, the least positive integer which has not appeared until then in $S$ is denoted by $\ell_S$.

Extend $S$ by rewriting the main terms after the initial block, term by term and in each step insert $\ell_S$ after each main term except $n$, i.e., we have
\beq
S=(1,2,3,\ldots,n,1,n+1,2,n+2,3,n+3,\ldots,n-1,2n-1,n).
\eeq
We start with $A=(1,2,3,4)$ and we have $A=(1,2,3,4,1,5,2,6,3,7,4)$. The next $\ell_S$ is $2n$. Until now the sequence $S$ is extended two blocks. To construct $b$th block, $b\geq 3$, we proceed as follows. Constitute the sequence $t$ from the terms between $n-1$ and $n$ exclusive in the $(b-1)$th block of $S$ and add $1$ to all of them, and the sequence $t'$ from the terms between last $n+1$ and its previous $n$ term, exclusive (e.g., in sequence $A$ we have $t=(8)$ and $t'=(1)$). In other words, for each $k>2$,
$t=(S(I_k(n-1)+1)+1,\ldots,S(I_k(n)-1)+1)$ and
$t'=(S(I_{k-1}(n)+1),\ldots,S(I_{k-1}(n+1)-1)).$
\begin{prop}\label{7}
The $\ell_S$ and $1$ appear once in $t$ and $t'$, respectively, and all of other terms of $t$ and $t'$ are the same.
\end{prop}
\begin{proof}
For the sequence $t'$, from Properties (II) and (III), for any $k>1$, terms between $S(I_k(n))$ and $S(I_k(n+1))$ should turn to terms between $S(I_k(n-1))$ and $S(I_k(n))$ by lower-trimming, and should turn to terms between $S(I_{k-1}(n))$ and $S(I_{k-1}(n+1))$ by upper-trimming. So the terms between $S(I_2(n))$ and $S(I_2(n+1))$ are upper-trimmed to the terms between $S(I_1(n))$ and $S(I_1(n+1))$, which is only the $1$ term. In the case of the sequence $t$, only one $\ell_S$ is inserted between each main term, especially $n-1$ and $n$, by the construction. So there's one $1$ term in the subsequence $T'=(S(I_{k-1}(n)),\ldots,S(I_{k-1}(n+1)))$ and one $\ell_S$ term in the subsequence $T=(S(I_k(n-1)),\ldots,S(I_k(n)))$. The lower-trimming of $S$ turn the terms of $T'$ to that terms of $S$ in which the terms of $T$ will turn to them when $S$ is upper-trimmed. Therefore all elements of $T$ and $T'$ other than $\ell_S$ and $1$ are the same.
\end{proof}

We denote by $Indx_t(\ell_S)$ and $Indx_{t'}(1)$ the position of $\ell_S$ and $1$ in the sequences $t$ and $t'$, respectively. We constitute a new sequence $P$ from $t$ and $t'$ by merging them as follows. Write all terms of $t'$ into the $P$ except 1. If $Indx_t(\ell_S)=Indx_{t'}(1)$ then insert one of the subsequences $(1,\ell_S)$ or $(\ell_S,1)$ in the position $Indx_t(\ell_S)$ into the $P$ sequence. For example, in the case of $A=(1,2,3,4,1,5,2,6,3,7,4)$, the next $\ell_S$ will be 8 and $P$ would be $(1,8)$ or $(8,1)$. On the other hand, if $Indx_t(\ell_S)\neq Indx_{t'}(1)$ then insert the $\ell_S$ and $1$ simultaneously into the $P$ in the positions $Indx_t(\ell_S)$ and $Indx_{t'}(1)$, respectively. For example, if $t=(a_1,\ell_S,a_2,a_3,a_4)$ and $t'=(a_1,a_2,a_3,1,a_4)$ then $P=(a_1,\ell_S,a_2,a_3,1,a_4)$. Set $d=Indx_P(1)-Indx_P(\ell_S)$. In $P$, delete the term $1$ and all the terms after $1$ ($P$ may become a null sequence). To extend the sequence $S$, append $P$ to the end of $S$. After that we want to rewrite the terms of last block, term by term such that if the appended term is a main term $m$, then insert the new $\ell_S$ into the position $Indx_S(\ell_S)=Indx_S(m)-d$ if and only if $I_B(1)\leq Indx_S(\ell_S)\leq I_B(n)$. In The case of our example sequence $A$, $P=()$ and $d=-1$, so it is extended as
\begin{align*}
A=(&1,2,3,4\\
&1,5,2,6,3,7,4\\
&1,8,5,2,9,6,3,10,7,4).
\end{align*}
For the next extension, $t=(11,8)$ and $t'=(1,8)$. Let $P=(11,1,8)$ and so $d=1$. Deleting $1$ and the terms after $1$ from $P$ and extending the sequence $A$ results that
\begin{align*}
A=(&1,2,3,4,\\
&1,5,2,6,3,7,4,\\
&1,8,5,2,9,6,3,10,7,4,11,\\
&1,8,5,12,2,9,6,13,3,10,7,14,4).
\end{align*}
At this point $t=(11,8,15)$ and $t'=(11,1,8)$, so $Indx_{t}(\ell_S)\neq Indx_{t'}(1)$. Therefore $P=(11,1,8,15)$ and then $d=-2$. Deleting $1$ and the terms after $1$ from $P$ and appending it to the end of $A$ demonstrates
\begin{align*}
A=(&1,2,3,4,\\
&1,5,2,6,3,7,4,\\
&1,8,5,2,9,6,3,10,7,4,11,\\
&1,8,5,12,2,9,6,13,3,10,7,14,4,11,\\
&1,8,15,5,12,2,9,16,6,13,3,10,17,7,14,4).
\end{align*}
Note that since $d=-2$, $\ell_S$'s are inserted two positions after each main term except for $n$ term.

Now assume that we start with the latter possibility of initial segment, i.e., $S=(\underset{n}{\underbrace{1,1,\ldots,1}},2)$. In this case first we construct the doubly fractal sequence started by $S'=(1,2,\ldots,n)$ and then extend $S$ by computing $k$th term of $S$, $k>n+1$, by the formula
\beq
S(k)=\# \{s'_i \mid  s'_i=s'_k, 1\leq i\leq k \}.
\eeq
For example, beginning from $A'=(1,1,1,1,2)$ then the extended sequence will be
\begin{align*}
A'=(&1,1,1,1,\\
&2,1,2,1,2,1,2,\\
&3,1,2,3,1,2,3,1,2,3,1,\\
&4,2,3,1,4,2,3,1,4,2,3,1,4,1,\\
&5,3,1,4,2,5,3,1,4,2,5,3,1,4,2,5).
\end{align*}
\section{The Main Result}\label{Rel}
In this section we discuss the correspondence of the family of doubly fractal sequences and the the family of signature sequences. In \cite{4} C. Kimberling conjectured that these two families should be the same. Here we try to give a proof to this conjecture.
\begin{lem}\label{8}
The $n+2$ first terms of a doubly fractal sequence can be the initial segment of infinitely many signature sequences.
\end{lem}
\begin{proof}
For any real number $n-1\leq \theta \leq n$, we have the relations
\beq
1+\theta<2+\theta<\ldots<n+\theta\leq1+2\theta\leq n+1+\theta
\eeq
in the construction of its signature sequence. So the initial segment of the signature sequence of each $n-1 \leq \theta \leq n$ will be $(1,2,3,\ldots,n,1,n+1)$. In the case of $\frac{1}{n+1} \leq \theta \leq \frac{1}{n}$, the inequalities are in the form:
\beq
1+\theta<1+2\theta<\ldots<1+n\theta\leq2+\theta\leq 1+(n+1)\theta.
\eeq
\end{proof}
\begin{thm}\label{9}
Every doubly fractal sequence $S$ can be made by the construction procedure explained in \S\ref{The Construction}. Also every sequence made by that procedure is doubly fractal.
\end{thm}
\begin{proof}
The conclusion will be obtained by induction. Let $S$ be an arbitrary doubly fractal sequence. By Lemma \ref{4}, the initial block of every doubly fractal sequence is either $(1,2,3,...,n,1)$ or $(1,1,1,\ldots,1,2)$. This coincides with the initial assumption in our construction. Now suppose that until the $k$th occurrence of the integer $n$, the sequence $S$ coincides with the sequence obtained by $k$th extension steps of the procedure of \S\ref{The Construction}. Define by $\mathfrak{P}_k(n)=(S(I_k(n)),\ldots,S(I_{k+1}(n)-1))$ the {\it $k$-th part} of $S$. Since $S$ is assumed to be doubly fractal, then Property (II) implies that:
\begin{eqnarray}\label{part}
\wedge_{\mathfrak{P}_{k+1}(n)}=\mathfrak{P}_k(n).
\end{eqnarray}
So the first term of the $(k+1)$st part is $n$ and all other of the main terms must appear in $(k+1)$st part. Since the $k$th part contains a term $n+1$, so by Property (II), the $(k+1)$st part should contain $n+1$. Let $\pi$ denote the terms appeared between $n$ and $n+1$ in the $(k+1)$st part of $S$. From the construction rule of $t'$ in \S3, one $1$ appears between $n$ and $n+1$ in the $k$th part of $S$, so $1$ should also appear in $\pi$. Property (III) implies that all the terms of $\pi$ must be lower-trimmed to terms between $(k+1)$st term $n-1$ and $(k+1)$st term $n$ (which was the $t$ sequence formed in the procedure \S3). There is one $\mathcal L_S$ in $t$ in $k$th part, so $\mathcal{L}_S+1$ should appear in $\pi$. So the merging sequence of $t$ and $t'$ as described in the procedure \S3 will exactly be the same as the $\pi$ sequence and the distance of $1$ and $\mathcal L_S+1$ in $\pi$ is equal to the $d$ defined in \S3. The rule of allocating the new $\ell_S$'s in the construction procedure assures that the distance of them from the main terms satisfy the Properties (II) and (III).
\end{proof}
\begin{thm}[Main Theorem] \label{10}
Each sequence $S=(s_1,s_2,s_3,\ldots)$ constructed by the procedure explained in \S\ref{The Construction} is a signature one. In other words, every doubly fractal sequence is a signature sequence.
\end{thm}
\begin{proof}
If there exists a positive real number $\theta$ such that $S_\theta=S$, we are done. On the contrary, suppose there is not any real number $\theta$ such that its signature sequence coincides with $S$. Suppose $\{1,2,\ldots,n\}$ are the main terms of $S$. By Lemma \ref{8}, first $n+2$ terms of $S$ is the initial segment of infinitely many signature sequences. Let $S_m$, $m\geq n+2$, be the longest initial segment of $S$ which is the initial segment of some signature sequence $S'$. So $s_{m+1}\neq s'_{m+1}$. By Proposition \ref{dsprop}, $S'$ is also doubly fractal, therefore one could assign $1$ or $\ell_S$ in the construction step of $s_{m+1}$, by Theorem \ref{9}.
Without lose of generality assume that $s_{m+1}=1$. The case $s_{m+1}=\ell_S$ is similar. According to the doubly fractality of $S'$, the term after $\ell_S$ in $S'$ must be $1$. So $S'_{m+2}=(s_1,\ldots,s_m,\ell_S,1)$ is the initial segment of $S'$. Since $S'_{m+2}$ is the initial segment of the signature sequence $S'$, they satisfy the inequalities
\begin{eqnarray}\label{ineq1}
s'_1+a'_1\theta \leq s'_2+a'_2\theta \leq \ldots \leq s'_m+a'_m\theta \leq \ell_S+\theta \leq 1+a'_{m+2}\theta
\end{eqnarray}
in which for $1\leq i\leq m+1$, $a'_i$ is the number of iterations of $s'_i$ in $S'_i$. By the assumption, since $S_{m+1}$ is not the initial segment of any signature sequence, one of the inequalities
\begin{eqnarray}\label{ineq2}
s_1+a_1\theta \leq s_2+a_2\theta \leq \ldots \leq s_m+a_m\theta \leq 1+a_{m+1}\theta
\end{eqnarray}
must be unsatisfied. Since $a_1=a'_1, a_2=a'_2,\ldots, a_m=a'_m, a_{m+1}=a'_{m+2}$, unsatisfiability of  (\ref{ineq2}) contradicts the satisfiability of (\ref{ineq1}).
\end{proof}

\begin{biog}
\item[Matin Amini](amini@helli.ir)
\begin{affil}
National Organization for Development of Exceptional Talents (NODET), Tehran, Iran\\
amini@helli.ir
\end{affil}
\end{biog}
\begin{biog}
\item[Majid Jahangiri](jahangiri@ipm.ir)
\begin{affil}
School of Mathematics, Institute for Research in Fundamental Sciences (IPM), Tehran, Iran\\
jahangiri@ipm.ir
\end{affil}
\end{biog}
\vfill\eject
\end{document}